%% file: arxivSubmission.tex
\documentclass[12pt, oneside]{amsart}

\input{preamble.tex}
\usepackage[margin = 1in]{geometry}

\title{A central limit theorem for the giant in a stochastic block model}
\author{David Clancy, Jr.}
\date{\today}
\usepackage{verbatim}
\usepackage{todonotes}
\newcommand{\diag}{\operatorname{diag}}

\begin{document}

\begin{abstract}
We provide a simple proof for of the central limit theorem for the number of vertices in the giant for super-critical stochastic block model using the breadth-first walk of Konarovskyi, Limic and the author (2024). Our approach follows the recent work of Corujo, Limic and Lemaire (2024) and reduces to the classic central limit theorem for the Erd\H{o}s-R\'{e}nyi model obtained by Stepanov (1970). 
\end{abstract}

\maketitle

{\small \noindent\textbf{Keywords.} stochastic block model, central limit theorem, random graphs, Delta method.
}
\section{Introduction}

It is well-known that if $G(n,p)$ is the Erd\H{o}s-R\'{e}nyi random graph on $n$ vertices and $p = c/n$ then the largest connected component $\cC_n(1)$ satisfies
\begin{align}\label{eqn:LLNforER}
    n^{-1}\#\cC_n(1) \weakarrow \rho(\lambda)
\end{align}where $\rho(c)$ is the largest solution to $x\ge 0$ to $1-e^{-c x} = x$. See \cite{ER.60} or \cite[Chapter 4]{vanderHofstad.17}. In fact, a central limit theorem (CLT) is known for the Erd\H{o}s-R\'{e}nyi random graph. In \cite{Stepanov.70}, Stepanov proved
\begin{align}\label{eqn:cltforer}
    n^{1/2} \left(n^{-1}\#\cC_n(1) - \rho(c)\right) \weakarrow \mathcal{N}(0,\sigma^2(c))
\qquad\textup{where}\qquad
    \sigma^2(c) = \frac{\rho(c) (1-\rho(c))}{(1-c(1-\rho(c)))^2}.
\end{align}
See also \cite{BBF.00,BR.12,CLL.24,EFL.23,Pittel.90,Rath.18} for proofs.

We are interested in establishing an analogous result of the stochastic block model \cite{HLL.83} using a method similar to \cite{BR.12,CLL.24} More precisely, we let $\SBM_n(n_1,\dotsm,n_d, P)$ be the stochastic block model on $n = \sum_{j} n_j$ many vertices where the vertices are partitioned into $d$ classes $\cV_j$ of (respective) size $n_j$. The edges are included independently according to the rule that for all $u,v\in \cV_1\cup\dotsm\cup \cV_d$ 
\begin{equation*}
    \PR(u\sim v) = p_{i,j} = 1-\exp(-\kappa_{i,j}/n)\qquad\textup{ for all }u\in \cV_i, v\in \cV_j.
\end{equation*}
Label the connected components of $\SBM_n(n_1,\dotms, n_d,P)$ by $(\cC_n(l);l\ge 1)$ in decreasing order of cardinality. We write $\boldsymbol{\cC}_n(l)$ for the vector whose $k$th entry is $\#\cC_n(l)\cap \cV_k$. In \cite{BJR.07}, Bollob\'{a}s, Janson and Riordan obtained a weak law of large numbers for the random vector $\bcC_n(1)$ under very mild assumptions. In lieu of restating their assumptions, we state the assumptions that turn out to be sufficient for a central limit theorem:
\begin{assumption} \label{ass:1}
    The following hold:
    \begin{enumerate}
        \item[\textbf{(i)}] For all $j$, $n_j = \mu_j n + \beta_j n^{1/2} + o(n^{1/2})$ as $n\to\infty$ for some probability vector $\bmu = (\mu_1,\dotsm, \mu_d)\in (0,\infty)^d$ and some $\beta_j\in \R$. 
        \item[\textbf{(ii)}] For all $i,j$: $p_{i,j} = 1-\exp(-\kappa_{i,j}^{(n)}/n)$ where $K^{(n)} = (\kappa_{i,j}^{(n)};i,j\in[d])\in \R_+^{d\times d}$ is a symmetric $d\times d$ matrix and satisfies
        \begin{align*}
            \kappa_{i,j}^{(n)} = \kappa_{i,j}+ n^{-1/2}\lambda_{i,j} + o(n^{-1/2}). 
        \end{align*}
        \item[\textbf{(iii)}] The matrix $K = (\kappa_{i,j};i,j\in[d])$ is irreducible and the Perron-Frobenius eigenvalue of $KM$ where $M = \operatorname{diag}(\bmu)$ is $\lambda_1>1$.
    \end{enumerate}
\end{assumption}

These are natural conditions to place on $\SBM_n(n_1,\dotsm, n_d, P)$ by the results of \cite{BJR.07} Bollob\'{a}s, Janson and Riordan and the CLT in Theorem 2.2 of \cite{Puhalskii.05} for the Erd\H{o}s-R\'{e}nyi random graph. In fact, if we let $\brho = (\rho_i;i\in[d])$ be the unique strictly positive solution to        \begin{equation}\label{eqn:rhodef}
            1-\exp\left(-\sum_{j=1}^n \kappa_{i,j}\mu_j \rho_j\right) = \rho_i.
\end{equation} the the results of \cite{BJR.07} imply the following.
\begin{proposition}[Bollob\'{a}s, Janson and Riordan \cite{BJR.07}]\label{prop:BJR}
    Let $\cC_n(1)$ be largest connected component of $\SBM_n(n_1,\dotsm,n_d,P)$. Under Assumption \ref{ass:1},
    \begin{align*}
        n^{-1}\boldsymbol{\cC_n}(1) \weakarrow M\brho\in(0,\infty)^d
    \end{align*}
    where $\brho$ is as in \eqref{eqn:rhodef}. Moreover, $\max_{l\ge 2}\#\cC_n(l) = O_{\PR}(\log(n))$.
\end{proposition}

We will prove the following.
\begin{theorem}\label{thm:CLT}
    Under Assumption \ref{ass:1},
    \begin{align*}
        n^{1/2} &\left(n^{-1}K\bcC_n(1) - KM\brho\right) \weakarrow  J^{-1} K \bzeta - J^{-1} (KB+\Lambda M)\brho - \Lambda M \brho
    \end{align*}
    where $ J =  KM(I-\diag(\brho)) -I$, $B = \operatorname{diag}(\bbeta)$ and $\bzeta = (\zeta_1,\dotms,\zeta_d)^T$ are independent centered normal random variables with $\E[\zeta_j^2] = \mu_j \rho_j(1-\rho_j)$.
\end{theorem}

\subsection{Discussion}

\subsubsection{Gaussian fluctuations for the giant}

Neal \cite{Neal.06} studies an epidemic model on a population partitioned into classes. He shows that under fairly general conditions that if the proportion of individuals infected is asymptotically positive, then the vector of infected individuals of each type is asymptotically normal. One can reformulate his result into random graphs (see Section 7 therein or \cite{Neal.03}) to recover Theorem \ref{thm:CLT} whenever $\lambda_{i,j} = \beta_j = 0$.

In \cite{Rath.18}, R\'{a}th provides an explicit formula for the generating function for the size of the component containing the vertex $1$ in the Erd\H{o}s-R\'{e}nyi random graph and uses that to obtain a new proof of \eqref{eqn:cltforer}. In \cite[Remark 2.2(ii)]{Rath.18}, he describes how to extend the generating function for the Erd\"{o}s-R\'{e}nyi random graph to the stochastic block model. In principle, this representation could be used to obtain a proof of the CLT in Theorem \ref{thm:CLT} (perhaps under the additional assumption that $\lambda_{i,j} = \beta_j = 0$); however, the analysis seems difficult.

Recently, in \cite{BBS.24} Bhamdi, Budhiraja and Sakanaveeti established a functional CLT for the size (i.e. $\|\bcC_n(1)\|_1$) of the giant for stochastic block models of finite type. Therein, they show that the limiting fluctuations are Gaussian involving an infinite collection of SDEs under the same assumptions as Theorem \ref{thm:CLT}. They do not give a concise formula for the covariance at a single time except in the Erd\H{o}s-R\'{e}nyi case. See Remark 3 therein.  

\subsubsection{Recovering \eqref{eqn:cltforer}}

Let us consider the $d = 1$ case. In this case $n_1 = n, \beta = 0$. If $\kappa^{(n)} = c + n^{-1/2}(\lambda+o(1))$ where $c>1$, then $J = c - 1 - c\rho = -(1-c(1-\rho))$ where we write $\rho = \rho(c)$. It follows that
\begin{align*}
        &-K^{-1} J^{-1} K\zeta = \frac{1}{(1-c(1-\rho))^2} \zeta \sim \mathcal{N}(0,\sigma^2(c))\\
        &- KJ^{-1} (KB+\Lambda M)\rho - K^{-1}\Lambda M\rho = \frac{\lambda\rho}{c(1-c(1-\rho))} - \frac{\lambda\rho}{c}= \frac{\lambda \rho(1-\rho)}{1-c(1-\rho)}
    \end{align*} Hence, Theorem \ref{thm:CLT} reduces to the result of Stepanov \cite{Stepanov.70} mentioned above when $d = 1$ and $\lambda = 0$. When $\lambda \neq 0$, this is the same limit as obtained in Theorem 2.2(b) of \cite{Puhalskii.05}.

\section{Size of the giant}\label{sec:BJR}

In \cite{BJR.07}, Bollob\'as, Janson and Riordan lay out a general theory for identifying the existence, uniqueness and asymptotic size of a giant connected component. We will recall two operators on functions (i.e. column vectors) $f:[d]\to\R$ defined therein. When writing a function $f$ as a column vector we will write $\bfl$. The two operators are
\begin{align}\label{eqn:TKdef}
    T_Kf(i) = \sum_{j=1}^d \kappa_{i,j} f(j) \mu_j = (KM \bfl)_i
\qquad\textup{and}\qquad 
    \Phi_Kf(i) = 1-\exp(-T_Kf(i)).
\end{align}

\begin{lemma}[Lemmas 5.8, 5.10 in \cite{BJR.07}]\label{lem:5.10BJR} Suppose that $KM\in \R_+^{d\times d}$ is irreducible, then there are at most two functions $f$ such that $\Phi_Kf = f$. Moreover, the solutions are either identically 0 or strictly positive everywhere.
\end{lemma}

One of the solutions is $0$ and we denote by $\rho:[d]\to\R_+^d$ (or $\brho$ in vector form) the unique largest solution to $\Phi_Kf = f$. Lemma 5.15 in \cite{BJR.07} gives precise conditions under which $\rho>0$ whenever we have a finite type case and Theorem 6.17 therein provides needed information on the matrix $KM(I-\diag(\brho))$ whenver $\brho$ is non-trivial. In combination with the Perron-Frobenius theorem (see, e.g. \cite{Woess.09}, Theorem 3.35)  one can easily show the following:
\begin{lemma} \label{lem:BJR_Lem2}Suppose that $KM$ is irreducible. The function $\rho:[d]\to(0,\infty)^d$ if and only if $KM$ has Perron-Frobenius eigenvalue $\lambda_1>1$. Moreover, the Perron-Frobenius eigenvalue $\lambda_1^*$ of $KM(1-\diag(\brho))$ is strictly smaller than 1, i.e. $\lambda_1^*<1$.
\end{lemma}

The second statement above implies that the matrix $J$ from Theorem \ref{thm:CLT} is invertible. 

\section{Fluctuations of the giant}\label{sec:fluctations}

\subsection{Preliminaries}

We will recall the construction of \cite{CKL.22} for encoding of the degree corrected stocahstic block model. To simplify notation, in this subsection we will write $\kappa_{i,j}$ instead of $\kappa_{i,j}^{(n)}$ and we will assume that $\kappa_{i,j}^{(n)}>0$. This can be done without loss of generality by considering a perturbation sufficiently small and using the general equivalence of Janson \cite{Janson.10a}.

In \cite{CKL.22}, the authors provided an encoding of the degree-corrected stochastic block model, which is an inhomogeneous version of the graph $\SBM_n(n_1,\dotsm, n_d,P)$. The degree-corrected stochastic block model consists of $n = \sum_{j=1}^d n_j$ many vertices  partitioned into $d$ classes $\cV_1\cup\dotsm\cup\cV_d$ where each vertex $u$ has a weight $w_u>0$ and edges are included independently with probability:
\begin{equation*}
\PR(u\sim v) = 1-\exp(-q_{i,j} w_u w_v) \qquad\textup{for all }\quad u\in \cV_i, v\in \cV_j.
\end{equation*}
Here $Q = (q_{i,j})$ is a $d\times d$ symmetric matrix with strictly positive diagonal. In the present work, we deal with the simple case where all the weights $w_{u} \equiv n^{-2/3}$ and $Q = (q_{i,j} = \kappa_{i,j}n^{1/3})$. This corresponds to the edge probabilities $P$ under Assumption \ref{ass:1}. Define
\begin{equation*}
    N_j^{(n)}(t) = \sum_{l=1}^{n_j} n^{-2/3} 1_{[\xi_{l,j}^\circ\le n^{1/3} \kappa_{j,j}t]} \qquad\textup{ where}\qquad \xi_{l,j}^\circ \sim \Exp(n^{-2/3})
\end{equation*}
and set
\begin{equation*}
    X_{i,j}^{(n)}(t_j) = \kappa_{i,i}^{-1} \kappa_{i,j} N_j^{(n)}(t_j) - 1_{[i=j]} t_j.
\end{equation*} The processes $X_{i,j}^{(n)}$ are precisely those studied in \cite{CKL.22}. For $\bt\in \R_+^d$, we write $\bbX^{(n)}(\bt) = (X_{i,j}^{(n)}(t_j))$ as the matrix valued function $\bbX^{(n)}:\R_+^d\to \R^{d\times d}$. We write $\bbX^{(n)}(\bt-) = (X_{i,j}^{(n)}(t_j-))$ as the $d\times d$ matrix. Note that for the vector $\bone\in \R^d$ of all $1$s we have
\begin{align*}
    \left(\bbX^{(n)}(\bt)\bone \right)_i= \sum_{j=1}^d X_{i,j}^{(n)}(t_j).
\end{align*}

Fix a vector $\bv\in (0,\infty)^d$ and let
\begin{equation*}
   \bT(y,\bv ,\bbX^{(n)}) =  \inf\{\bt: \bbX^{(n)}(\bt-) \bone = -\bv y\}.
\end{equation*} That is, $\bT = \bT(y,\bv,\bbX^{(n)})$ is a solution to 
\begin{align*}
    \bbX^{(n)}(\bt-)\bone = -\bv y
\end{align*}and any other solution $\bt'$ satisfies $\bT\le \bt'$ coordinate-wise. These are well-defined by the construction of \cite{Chaumont:2020} and are a.s. finite since $N_j^{(n)}$ are bounded for each fixed $n$. A key result in \cite{CKL.22} is that
\begin{equation}\label{eqn:hittingTimesX}
    \bT(y,\bv,\bbX^{(n)}) = \bv y + \sum_{l: E_{l}<y} \bDelta^{\bbX^{(n)}}{(l)}
\end{equation}
where $E_{l}$ are the random jump times of $y\mapsto \bT(y,\bv,\bbX^{(n)})$ arranged in some order. Moreover, one can label $\bDelta^{\bbX^{(n)}}(l)$ in a random way so that
\begin{equation}\label{eqn:Xhittingtimejumps1}
    (\bDelta^{\bbX^{(n)}}(l);l\ge 1) \overset{d}{=}\left( n^{-2/3}\operatorname{diag}(K)^{-1} K \boldsymbol{\cC}_n(l)\right).
\end{equation}

More precisely, we have the following theorem.
\begin{theorem}\label{thm:CKL1} Fix a vector $\bv\in (0,\infty)^d$. 
Let $\cC_n(l)$ be the connected components of $\G_n\sim \SBM(n_1,\dotsm,n_d,P)$ where $p_{i,j} = 1-e^{-\kappa_{i,j}/n}$ and $\kappa_{i,i}>0$ for all $i$. Conditionally given the connected components $\cC_n(l)$, generate independent exponential random variables $E_l\sim \Exp(n^{-1/3}\bv^T \operatorname{diag}(K)\bcC_n(l))$. Then
\begin{align*}
    \left(\bT(y,\bv,\bbX^{(n)}) ;y\ge 0\right)\overset{d}{=} \left(\bv y + \sum_{l: E_l<y} n^{-2/3} \operatorname{diag}(K)^{-1} K \bcC_n(l);y\ge 0\right).
\end{align*}
\end{theorem}

\subsection{A change of variable}

It is actually easier to deal with a slightly modified process $\Z^{(n)}(\bt) = (Z_{i,j}^{(n)}(t_j))$ defined by
\begin{equation*}
Z_{i,j}^{(n)}(t) = n^{-1/3}\kappa_{i,i} X_{i,j}^{(n)}( n^{1/3} \kappa_{j,j}^{-1} t_j)
\end{equation*} which, as a matrix equation, becomes
\begin{equation*}
    \Z^{(n)}(\bt) = n^{-1/3} \operatorname{diag}(K) \bbX^{(n)}\left(n^{1/3}\operatorname{diag}(K)^{-1}\bt\right).
\end{equation*}
Observe that for any $\bv\in (0,\infty)^d$
\begin{align*}
    \bT(y,\bv, \Z^{(n)}) &= \inf\{\bt\in \R_+^d: n^{-1/3} \operatorname{diag}(K) \bbX^{(n)}(n^{1/3} \operatorname{diag}(K)^{-1}\bt) = -y\bv\}\\
    &= \inf\{n^{-1/3}\diag(K)\bs\in \R_+^d: \bbX^{(n)}(\bs) = - n^{1/3}y\operatorname{diag}(K)^{-1} \bv\}
    \\& = n^{-1/3} \operatorname{diag}(K)\bT\left(n^{1/3} y,\operatorname{diag}(K)^{-1}\bv, \bbX^{(n)} \right).
\end{align*} The second inequality follows from $\bt = n^{-1/3} \operatorname{diag}(K)\bs$ and the last equality follows from the left-continuity of $y\mapsto\bT(y,\bv,\Z^{(n)})$. 
Therefore, using \eqref{eqn:hittingTimesX} and \eqref{eqn:Xhittingtimejumps1}, the jumps of $y\mapsto \bT(y,\bv, \Z^{(n)})$ can be represented as
\begin{align*}
    \bT(y,\bv, \Z^{(n)}) = \bv y + \sum_{l: E'_l < y} \bDelta_n(l)\qquad\textup{ where}\qquad \bDelta_n(l) = n^{-1} K \boldsymbol{\cC}_n(l).
\end{align*}
A corollary of Theorem \ref{thm:CKL1} and standard properties of exponential random variables is the following.
\begin{corollary}\label{cor:bTbhave}
    Let $\bv\in (0,\infty)^d$ be fixed and let $\cC_n(l)$ be the connected components of $\SBM_n(n_1,\dotms,n_d,P)$ where $p_{i,j} = 1-e^{-\kappa_{i,j}/n}$ for all $i,j$ and $\kappa_{i,i}>0$ for all $i$. Conditionally given $\cC_n(l)$, let $E_l'\sim \Exp(\bv^T \bcC_n(l))$. Then
    \begin{equation*}
        \left(\bT(y,\bv,\Z^{(n)});y\ge 0\right) \overset{d}{=}\left(\bv y + \sum_{l:E_l'<y} n^{-1} K \bcC_n(l);y\ge 0\right).
    \end{equation*}
\end{corollary}

\subsection{Path behavior of $Z_{i,j}^{(n)}$}

We now write out more explicitly the dependence of the entries $\kappa_{i,i}$ on $n$. Observe that
\begin{align}\label{eqn:zform1}
    Z_{i,j}^{(n)}(t_j) = - 1_{[i=j]} t_j+\kappa_{i,j}^{(n)}\sum_{j=1}^{n_j} n^{-1} 1_{[\xi_{l,j}\le t]}  \quad\textup{where}\quad \xi_{l,j} = n^{-2/3} \xi_{l,j}^\circ \overset{d}{=}\Exp(1).
\end{align}
The summation in \eqref{eqn:zform1} is essentially the empirical distribution function of $n_j$ many i.i.d. $\Exp(1)$ random variables. Donsker's theorem therefore implies for all $j$
\begin{align}\label{eqn:donsker1}
    \left(\sqrt{n_j}\left(n_j^{-1} \sum_{j=1}^{n_j} 1_{[\xi_{l,j}\le t]} - (1-e^{-t}) \right);t\ge0\right)\weakarrow \left(B_j^{\operatorname{br}}(1-e^{-t});t\ge 0\right)
\end{align}
in the Skorohod space. 
 Under Assumption \ref{ass:1}, we can easily check \eqref{eqn:donsker1} implies the following lemma. 
\begin{lemma}\label{lem:Donsker}
Suppose Assumption \ref{ass:1}. Jointly for all $i,j$, 
\begin{align*}
    &\left( n^{1/2}\bigg( Z_{i,j}^{(n)}(t) - \kappa_{i,j}\mu_j (1-e^{-t}) + 1_{[i=j]}t\bigg);t\ge 0\right)
    \\ &\qquad \weakarrow \left((\kappa_{i,j}\beta_j+\lambda_{i,j}\mu_j )(1-e^{-t}) + \kappa_{i,j}\sqrt{\mu_j} B_j^\br(1-e^{-t});t\ge0 \right)
\end{align*} in the Skorohod space 
where $B_j^\br$ are independent standard Brownian bridges.   
\end{lemma}

As Proposition \ref{prop:BJR}  has a deterministic weak limit, we see
\begin{corollary}\label{cor:jointConv}
    The convergence in Proposition \ref{prop:BJR} and Lemma \ref{lem:Donsker} hold jointly.
\end{corollary}

It will be convenient to reformulate Lemma \ref{lem:Donsker} in terms of a vector-valued process $\bZ^{(n)}$. For each $i\in [d]$, let 
\begin{align*}
    Z_i^{(n)} (\bt) = \sum_{j=1}^n Z_{i,j}^{(n)}(t_j)
\qquad\textup{and}\qquad
    \varphi_i(\bt) = -t_i + \sum_{j=1}^n \kappa_{i,j}\mu_j(1-e^{-t_j}).
\end{align*}
By using the Skorohod representation theorem we have
\begin{align}
  \nonumber  Z_i^{(n)}(\bt) = \varphi_i(\bt) 
  &+n^{-1/2}\left( \sum_{j=1}^d (\kappa_{i,j}\beta_j+\lambda_{i,j}\mu_j)(1-e^{-t_j})+  \sum_{j=1}^d \kappa_{i,j}\sqrt{\mu_j} B_j^\br(1-e^{-t_j}) +o(1)\right) \\
    &= \varphi_i(\bt) + n^{-1/2}\left(m_i(\bt)+ \Psi^\circ_i(\bt)\right) + \delta_{i}(\bt) \qquad\textup{(say)} \label{eqn:bZform}.
\end{align} Here $n^{1/2}\delta_{i}(\bt) \to 0$ locally uniformly in $\bt\in\R_+^d$ is just an explicit representation of the error term.
Note that $\Psi^\circ_i:\R_+^d\to \R$ are continuous Gaussian processes with mean $0$. 
We abbreviate $m_i+\Psi_i^\circ = \Psi_i$. Let us also write $\bphi = (\varphi_1,\dotms, \varphi_d)^T$ and $\bPsi(\bt)= (\Psi_1(\bt),\dotms, \Psi_d(\bt))^T$. Similarly denote $\bPsi^\circ$. While the precise covariance structure is not needed, we note that at $\bt = KM\brho$ that
\begin{equation}\label{eqn:KMPsival}
-\bPsi(KM\brho) \overset{d}{=} \left(KB+\Lambda M\right) \brho + K\bzeta
\end{equation} where $B$ and $\bzeta$ are as in Theorem \ref{thm:CLT}.

\subsection{Path properties of $\bphi$}

We first note that if we set $f(j) = 1-e^{-t_j}$ have $t_i = -\log(1-f(i))$ and so
\begin{equation*}
    \varphi_i(\bt) = -t_i +\sum_{j=1}^d \kappa_{i,j}\mu_j (1-e^{-t_j}) = T_Kf(i) - \log(1-f(i))
\end{equation*}
where $T_K$ is introduced in \eqref{eqn:TKdef}. Rearranging, we see that there is a bijection between
\begin{align*}
    \left\{\bt\in \R_+^d: \bphi(\bt) = \bzer \right\} \qquad\textup{and}\qquad  \{\rho:[d]\to\R_+^d: \Phi_K\rho = \rho\}.
\end{align*} 
By Lemma \ref{lem:5.10BJR} and Lemma \ref{lem:BJR_Lem2}, Assumption \ref{ass:1} implies there are precisely two values $\bt\in \R^d_+$ on the left-hand side above. We label these two points as $\bzer$ and $\bt^0$. Note that
\begin{equation*}
    1-e^{-t_i^0} = \rho_i.
\end{equation*}
Since we know that $1-e^{-T_K \rho} = \rho$, we get the $t_i^0 = T_K \rho(i).$ Equivalently, $\bt^0 = KM\brho$.

The next lemma gives a way to ``construct'' $\bt^0$ using the first hitting times of \cite{Chaumont:2020}.
\begin{lemma}
Let $\ba$ be the (right) Perron-Frobinous eigevnector of $KM$ normalized so that $\sum_{i} a_i \mu_i = 1$. Then under Assumption \ref{ass:1},
\begin{equation*}
    \bt^{0} = \lim_{y\downarrow0} \inf\{\bt\in \R_+^d : \bphi(\bt) = -y\ba\}.
\end{equation*}
\end{lemma}
\begin{proof}
Define $\bv^y = \inf\{\bt\in \R_+^d: \bphi(\bt)=-y\ba\}$. By \cite{Chaumont:2020}, $\bv^y<\bv^z$ coordinate-wise for all $y<z$ and so $\bv^0:= \lim_{y\downarrow0} \bv^y$ exists. Note that since $\bphi$ is continuous, $\bphi(\bv_0) = \bzer$. It suffices to show that $\bv_0\neq \bzer$. This is relatively easy. 

To start, note that at least one coordinate of $\bv^y$ for $y>0$ is strictly positive as $\bphi(\bzer) = \bzer$ and $\bphi$ is continuous. Suppose $J\subset [d]$ are all coordinates of $\bv^y$ that are strictly positive. If $J\neq [d]$, then for any $i\notin J$, we have
    \begin{equation*}
        \varphi_i(\bv^y) =-v^y_{i} +  \sum_{k\notin J} \kappa_{i,k}\mu_k(1-e^{-v^y_{k}}) + \sum_{j\in J} \kappa_{i,j}\mu_j (1-e^{-v^y_{j}}) = \sum_{j\in J} \kappa_{i,j}\mu_j (1-e^{-v^y_{j}}) >0
    \end{equation*} contradicting $\bphi(\bv^y) = -y\bone$. Hence $\bv^y\in (0,\infty)^d$. Now suppose that $\bv^y\to \bzer$. One can easily check the Jacobian of $\bphi:\R^d\to\R^d$ at $\bzer$ is
\begin{equation*}
  J_{\bphi}(\bzer) =(
      \kappa_{i,j}\mu_j - 1_{[i=j]}
  ;i,j\in[d])= KM-I.
\end{equation*}
Hence, as $y\to 0$ we have $\bphi(\bv^y) = (KM-I)\bv^y + O(\|\bv^y\|^2)$ and this implies
\begin{equation}\label{eqn:pfeqn1}
    -y \ba = (KM-I)\bv^y + O(\|\bv^y\|^2).
\end{equation}
Rearranging, we see with the coordinate-wise comparison
\begin{equation*}
    KM \bv^y = \bv^y - y \ba + O(\|\bv^y\|^2) \le \lambda_1 \bv^y\qquad\forall y\textup{ small enough}.
\end{equation*}
Hence, by the Perron-Frobenius theorem (see, e.g. \cite[Theorem 3.35(ii)]{Woess.09}), $\bv^y = c_y \ba$ for some constant $c_y> 0$. This contradicts \eqref{eqn:pfeqn1} and so $\bv^y$ does not converge to zero.
\end{proof}

\subsection{Estimate for the long excursion interval}

By Proposition \ref{prop:BJR} there exists a single connected component $\cC_n(1)$ that is of size $\Theta(n)$ while all the remaining components are of order $O(\log n)$. As commented above, $ \bDelta_n(1) \overset{d}{=} n^{-1} K^{(n)} \bcC_n(1)$
is the largest jump of $y\mapsto \bT(y,\ba, \Z^{(n)})$. Let us write $\bL$ and $\bR$ as the ``left'' and ``right'' endpoints of this jump. That is $\bR_n = \bT(Y_0+, \ba,\Z^{(n)})$, $\bL_n = \bT(Y_0, \ba,\Z^{(n)})$ where $Y_0$ is the location of the largest jump of $\bT$. This implies
\begin{equation*}
    \bDelta_n(1) = \bR_n - \bL_n = \bT(Y_0+, \ba,\Z^{(n)}) - \bT(Y_0,\ba,\Z^{(n)})
\end{equation*}
In fact, by Corollary \ref{cor:bTbhave}, we know that $Y_0  = E_1'$ where, conditionally given $\bcC_n(l)$, $E_1'\sim \Exp(\ba^T \bcC_n(1))$.

As in \cite{CLL.24}, we need to estimate the value of $\bL_n$ and $\bZ^{(n)}(\bL_n-)$. By Corollary \ref{cor:bTbhave}, we know that, conditionally given $(\bcC_n(l);l\ge 1)$, that
\begin{equation*}
    \bL_n
    \big| (\bcC_n(l);l\ge 1)\overset{d}{=} -E_1' \ba + \sum_{l: E_l'<E_1'} n^{-1}K^{(n)} \bcC_n(l)
\end{equation*}
where $E_l'\sim \Exp(\ba^T\bcC_n(l))$ are conditionally independent. Also, coordinate-wise
\begin{equation}\label{eqn:zcomp}
   \bzer \ge \bZ(\bL_n-) \ge  -\bL_n 
\end{equation}
since $\bZ(\bL_n-) = -y\ba$ for some $y\ge 0$ (the definition of first hitting time) and $Z_{i,j}(\bt)\ge -t_i$.

We now prove the following lemma.
\begin{lemma}\label{lem:znLnandLn}
    Under Assumption \ref{ass:1}. the following convergences hold
    \begin{align*}
        &n^{1/2}\bL_n\weakarrow \bzer,& n^{1/2}\bZ^{(n)}(\bL_n-)\weakarrow\bzer,&
    &n^{1/2}\bZ^{(n)}(\bR_n)\weakarrow \bzer,&  &\bR_n \weakarrow KM\brho&.
    \end{align*}
    Moreover, they hold jointly with the convergence in Corollary \ref{cor:jointConv}.
\end{lemma}

The proof is actually a consequence of Proposition \ref{prop:BJR}, Corollary \ref{cor:bTbhave} and the following elementary lemma.
\begin{lemma}\label{lem:help11}
    Let $X^{(n)}_0\ge X^{(n)}_1\ge \dotms\ge X_n^{(n)}\ge 0$ be a sequence of random variables and $\bV_0^{(n)},\dotsm \bV_n^{(n)}$ be a sequence of random vectors in $\R^d$ for some $d$. Suppose that
    \begin{enumerate}
        \item There is some $\eps>0$ and constant $C$  such that 
        \begin{equation*}
            \PR\left(X_0^{(n)}> \eps n, \sum_{j=1}^n X_j^{(n)} \le C n\right) \to 1
        \end{equation*}
        \item For some $\omega(n)\to\infty$ with $\omega(n) = o(n^{1/2})$ it holds $\PR(X_1^{(n)} \le \omega(n)) \to 1.$
        \item There is a constant $\alpha\ge 0$ such that $\|\bV_j^{(n)}\|\le \alpha X_j^{(n)}$ for all $j$.
    \end{enumerate}
    Conditionally given $(X_0^{(n)},\dotsm, X_n^{(n)})$ let $\xi_l^{(n)}\sim \Exp(X_l^{(n)})$ be conditionally independent exponential random variables. Then, jointly, 
     \begin{equation*}
      \bX_n:=  \sqrt{n}\sum_{l: \xi_l^{(n)}< \xi_0^{(n)}} n^{-1}\bV^{(n)}_l \weakarrow \bzer\qquad\textup{and}\qquad \sqrt{n}\xi_0^{(n)}\weakarrow 0.
    \end{equation*}
   
\end{lemma}
\begin{proof}
    The second is an obvious consequence of assumption (1), so we only prove the statement about $\bV^{(n)}_l$. Without loss of generality we assume that $\alpha = 1$.

     Let $S_n = \#\{l:\xi_l^{(n)}<\xi_0^{(n)}\}$. Note that by (2) and (3), we know $\|\bX_n\|\le n^{-1/2}\omega(n) S_n$. Hence it suffices to show that $\PR\left(S_n> n^{1/2}/ \omega(n)\right) \to 0.$ By elementary properties of exponentials, $\PR(\xi_l^{(n)}< \xi_0^{(n)}|(X_j^{(n)};j)) \le X_l^{(n)}/X_0^{(n)}$ for all $l\ge 1$. Therefore, Markov's inequality implies
   \begin{align*}
       \PR\left(S_n > z| (X_j^{(n)};j)\right) \le \frac{1}{z} \sum_{l=1}^n \frac{X_l^{(n)}}{X_0^{(n)}}.
   \end{align*}
   Let $\mathcal{E}_n$ be the event described in hypothesis (1). Then $\sum_{l=1}^n \frac{X_l^{(n)}}{X_0^{(n)}}1_{[\mathcal{E}_n]} \le \frac{C}{\eps}$ and so
   \begin{align*}
       \PR&\left(S_n > n^{1/2}/\omega(n)\right) \le \PR(\mathcal{E}_n^c) + \PR\left(S_n > n^{1/2}/\omega(n) | \mathcal{E}_n\right)\le o(1)+ \frac{C}{\eps}\frac{\omega(n)}{n^{1/2}} = o(1).
   \end{align*} 
\end{proof}

\begin{proof}[Proof of Lemma \ref{lem:znLnandLn}]
Set $X_l^{(n)} = \ba^T \cC_n(l+1)$ and extending the list by zeros to be length $n$. Hypothesis (1) and (2) of Lemma \ref{lem:help11} are a reformulation of Proposition \ref{prop:BJR} with $\omega(n) = \log(n)^2$, say and $\xi_l^{(n)}$ are the random variables in Corollary \ref{cor:bTbhave}. Lastly, if we set $\bV^{(n)}_l = K^{(n)} \bcC_n(l+1)$, since $\ba\in(0,\infty)^d$ we clearly have $\|\bV^{(n)}_l\|\le  \alpha \ba^T \cC_n(l+1)$ for some $\alpha$ large enough. 
Lastly, 
\begin{equation*}
    \bL_n = -E_1' \ba^T + \sum_{l: E_l'<E_1'} n^{-1} K^{(n)}\bcC_n(l) = o_\PR(n^{-1/2}) 
\end{equation*}
by Lemma \ref{lem:help11}. The convergence of $\bZ^{(n)}$ follows from the inequalities \eqref{eqn:zcomp}.

The results for $\bR_n$ follow immediately from the definition of first hitting times, Proposition \ref{prop:BJR}, and the convergences of $\bL_n$ and the fact that all limits are deterministic.
\end{proof}

\subsection{Proof of the CLT}

Since $\bDelta_n(1) = \bR_n-\bL_n = n^{-1}K^{(n)}\bcC_n(1))$, we have 
\begin{align*}
    n^{1/2}&\left(n^{-1}K^{(n)} \bcC_n(1) - KM\brho \right) = n^{1/2}\left(\bR_n - \bL_n -K M\brho\right)= n^{1/2}( \bR_n - \bt^0) - n^{1/2}\bL_n
\end{align*} where the last equality uses $\bt^0 = KM\brho$.

Using Skorohod representation and the expansion $\bZ = \bphi + n^{-1/2}\bPsi + \bdelta$ we have
\begin{align*}
    \sqrt{n} \bphi(\bR_n) &= \sqrt{n}\bphi(\bR_n)  - \sqrt{n} \bphi(\bt^0)\\
    &= \sqrt{n} \bZ^{(n)}(\bR_n) - \bPsi(\bR_n) - \bdelta^{(n)}(\bR_n)- \sqrt{n}\bZ^{(n)}(\bt_0) + \bPsi(\bt^0) + \bdelta^{(n)}(\bt^0)\\
    &= - \sqrt{n} \bZ^{(n)}(\bt^0) + \sqrt{n}\bZ^{(n)}(\bR_n) + \bPsi(\bt^0)-\bPsi(\bR_n) + n^{1/2}\bdelta^{(n)}(\bt^0) - n^{1/2}\bdelta^{(n)}(\bR_n).
\end{align*}
Note $n^{1/2}\bdelta^{(n)}\to \bzer$ locally uniformly, $\bR_n\to\bt^0$, $\bPsi$ is continuous and $\sqrt{n}\bZ^{(n)}(\bR_n) \to \bzer$ by Lemma \ref{lem:znLnandLn} (and our application of Skorohod's representation theorem). Hence,
\begin{equation*}
    \sqrt{n} \bphi(\bR_n) = -\sqrt{n} \bZ^{(n)}(\bt^0) + o(1).
\end{equation*}
By Lemma \ref{lem:Donsker}, and the fact that $\bphi(\bt^0) = \bzer$, we have shown (using our application of Skorohod)
\begin{equation}\label{eqn:phiconvatRn}
    \sqrt{n} \bphi(\bR_n) \longrightarrow -  \bPsi(\bt^0).
\end{equation}

The remainder of the proof is essentially the Delta method of asymptotic statistics \cite[Chapter 3]{vanderVaart.00} and the inverse function theorem. Let us compute the Jacobian of $\bphi$:
\begin{align*}
    J_{\bphi}(\bt) =(
        \frac{\partial}{\partial t_j} \varphi_i(\bt)
    ;i,j\in[d])= (-1_{[i=j]}+\kappa_{i,j} \mu_j e^{-t_j};i,j).
\end{align*} 
Evaluating at $\bt^0$, we have
\begin{align*}
    J_{\bphi}&(\bt^0) = (-1_{[i=j]} + \kappa_{i,j} \mu_j e^{-t_j^0};i,j) = -I + K M\operatorname{diag}(e^{-t^0_j};j\in[d]) \\
    &= KM-I - KM\operatorname{diag}(\brho) = J
\end{align*}
where we used $1-e^{-t_j^0} = \rho(j)$. Note that this Jacobian is invertible by Lemma \ref{lem:BJR_Lem2}.

We now state the following lemma, whose proof is an easy application of the inverse function theorem.
\begin{lemma} 
    Suppose that $\bfl:\R^d\to \R^d$ be a smooth map and suppose that $A = J_{\bfl}(\bx)\in {GL}_d(\R)$ is invertible for some fixed $\bx$. If $\by_n\to \bx$ and $\sqrt{n} (\bfl(\by_n)-\bfl(\bx))\to \bz$ then
    \begin{equation*}
        \sqrt{n}(\by_n-\bx) \to A^{-1} \bz.
    \end{equation*}
\end{lemma}

Combining the above lemma with \eqref{eqn:phiconvatRn}, we have shown (under our a.s. coupling that)
    $n^{1/2}(\bR_n - \bt^0) \to - J^{-1} \bPsi(\bt^0).$ Since $\bL_n = o(n^{-1/2})$, we have shown
\begin{align*}
    n^{1/2}\left(n^{-1}K^{(n)} \bcC_n(1)-KM\brho\right) \to - J^{-1} \Psi(\bt^0).
\end{align*}
But the left-hand side is
\begin{align*}
    n^{1/2}\left(n^{-1} K\bcC_n(1) - KM\brho\right) + \Lambda M\brho + o(1)
\end{align*} where we use $n^{-1}\bcC_n(1) = M\brho + o(1)$ a.s. under our application of Skorohod's representation theorem. 
The result now easily follows from \eqref{eqn:KMPsival}.

\end{document}

%% file: preamble.tex
\usepackage{graphicx}

\usepackage[english]{babel}

\usepackage[colorlinks= True, citecolor=blue]{hyperref}
\usepackage{amsthm}
\usepackage{amsfonts}
\usepackage{amssymb}
\usepackage{amsmath}
\usepackage{mathrsfs}
\usepackage{mathtools}

\usepackage{fancyhdr}
\usepackage{xcolor}
\usepackage{ wasysym }

\usepackage{enumerate}

\newtheorem{theorem}{Theorem}[section]
\newtheorem{corollary}[theorem]{Corollary}
\newtheorem{lemma}[theorem]{Lemma}
\newtheorem{proposition}[theorem]{Proposition}

\theoremstyle{definition}

\newtheorem{assumption}[theorem]{Assumption}

\newcommand{\bone}{\mathbf{1}}
\newcommand{\ba}{\mathbf{a}}

\allowdisplaybreaks

\newcommand{\bzeta}{\boldsymbol{\zeta}}

\newcommand{\bzer}{\boldsymbol{0}}

\newcommand{\bz}{{\mathbf{z}}}
\newcommand{\bZ}{{\boldsymbol{Z}}}

\newcommand{\bbeta}{\boldsymbol{\beta}}

\DeclareMathAlphabet{\mathscrbf}{OMS}{mdugm}{b}{n}

\numberwithin{equation}{section}

\newcommand{\bcC}{{\boldsymbol{\cC}}}

\newcommand{\dotms}{\dotsm}

\newcommand{\by}{{\bf y}}
\newcommand{\bx}{{\bf x}}
\newcommand{\br}{{\operatorname{br}}}

\newcommand{\bphi}{\boldsymbol{\varphi}}

\newcommand{\bfl}{\boldsymbol{f}}
\newcommand{\bs}{\boldsymbol{s}}
\newcommand{\bv}{{\boldsymbol{v}}}
\newcommand{\bV}{{\boldsymbol{V}}}
\newcommand{\bPsi}{\boldsymbol{\Psi}}

\newcommand{\bt}{{\boldsymbol{t}}}

\newcommand{\bDelta}{\boldsymbol{\Delta}}

\newcommand{\eps}{\varepsilon}
\newcommand{\R}{\mathbb{R}}

\newcommand{\Z}{\mathbb{Z}}

\newcommand{\E}{\mathbb{E}}

\newcommand{\cV}{\mathcal{V}}

\newcommand{\bX}{{\boldsymbol{X}}}
\newcommand{\bT}{{\bf T}}
\newcommand{\bmu}{\boldsymbol{\mu}}
\newcommand{\brho}{\boldsymbol{\rho}}

\newcommand{\weakarrow}{{\overset{(d)}{\Longrightarrow}}}

\newcommand{\bL}{{\boldsymbol{L}}}
\newcommand{\bR}{{\boldsymbol{R}}}
\newcommand{\PR}{\mathbb{P}}
\newcommand{\bdelta}{\boldsymbol{\delta}}

\newcommand{\bbX}{{\mathbb{X}}}

\newcommand{\G}{{\mathcal{G}}}

\newcommand{\SBM}{{\operatorname{SBM}}}
\newcommand{\cC}{{\mathcal{C}}}

\newcommand{\Exp}{\textup{Exp}}